\documentclass[a4paper,12pt,reqno]{amsart}

\usepackage{amsthm}
\usepackage{amsmath}
\usepackage{amssymb}
\usepackage{mathrsfs}
\usepackage{latexsym}
\usepackage{exscale}
\usepackage{geometry}
\usepackage{graphicx}
\usepackage[utf8]{inputenc}
\usepackage{verbatim}

\usepackage{color}
\definecolor{red}{rgb}{1,0.1,0.1}
\definecolor{blue}{rgb}{0.1,0.1,1}
\definecolor{vb}{RGB}{160,32,240}

\theoremstyle{plain}

\newtheorem*{teo*}{Theorem}
\newtheorem*{prop*}{Proposition}

\numberwithin{equation}{section}
\newtheorem{teo}{Theorem}[section]
\newtheorem{lema}[teo]{Lemma}
\newtheorem{corol}[teo]{Corollary}
\newtheorem{prop}[teo]{Proposition}

\newtheorem{teoA}{Theorem}

\newtheorem{propA}[teoA]{Proposition}

\theoremstyle{remark}
\newtheorem{obs}[teo]{Remark}

\theoremstyle{definition}

\newtheorem*{mydef*}{Definition}
\newtheorem{mydef}[teo]{Definition}

\newtheorem{mydefA}[teoA]{Definition}

\calclayout

\allowdisplaybreaks

\newcommand{\R}{\mathbb{R}^n}
\newcommand{\C}{\mathbb{C}}
\newcommand{\Z}{\mathbb{Z}}
\newcommand{\N}{\mathbb{N}}

\newcommand{\A}{\mathcal{A}}
\newcommand{\B}{\mathcal{B}}

\begin{document}
\title[H\"ormander's conditions for vector-valued kernels]{ H\"ormander's conditions for vector-valued kernels of singular integrals and its commutators}

\author[A.~L.~Gallo]{Andrea L. Gallo}
\address{A.~L.~Gallo \\ FaMAF \\ Universidad Nacional de C\'ordoba \\
CIEM (CONICET) \\ 5000 C\'ordoba, Argentina}
\email{andreagallo88@gmail.com}

\author[G.~H.~Iba\~{n}ez~Firnkorn]{Gonzalo H. Iba\~{n}ez-Firnkorn}
\address{G.~H.~Iba\~{n}ez~Firnkorn\\ FaMAF \\ Universidad Nacional de C\'ordoba \\
CIEM (CONICET) \\ 5000 C\'ordoba, Argentina}
\email{gibanez@famaf.unc.edu.ar}

\author[M.~S.~Riveros]{Mar\'{\i}a Silvina Riveros}
\address{M.~S.~Riveros \\ FaMAF \\ Universidad Nacional de C\'ordoba \\
CIEM (CONICET) \\ 5000 C\'ordoba, Argentina}
\email{sriveros@famaf.unc.edu.ar}

\thanks{ The authors are  partially supported by
CONICET and SECYT-UNC}

\subjclass[2010]{42B20, 42B25}

\keywords{Calder\'on-Zygmund operators, commutators, BMO, H\"ormander's
condition of Young type, Muckenhoupt weights,
vector-valued inequalities.}


\begin{abstract}
In this paper we study Coifman type estimates and weighted norm
inequalities  for singular integral operators $T$ and its
commutators, given by the convolution  with  a vector valued kernel
$K$. We define a weaker H\"ormander type condition associated with
Young functions for the vector valued kernels.  With this  general
framework we obtain as an example the result for the square operator
and its commutator given in [ M. Lorente, M.S. Riveros, A. de la
Torre \emph{On the Coifman type inequality for the oscillation of
the one-sided averages }, Journal of Mathematical Analysis and
Applications, Vol 336, Issue 1,\ (2007) 577-592.]

\end{abstract}

\maketitle


\section{Introduction}

For several years, a classical problem in the harmonic analysis is
the following:
given a linear operator $\mathcal{T}$, find the maximal
operator $\mathcal{M}_{\mathcal{T}}$ such that $\mathcal{T}$ is controlled by
$\mathcal{M}_{\mathcal{T}}$ in the following sense,
\begin{align}\label{Coifineq}
\int_{\R} |\mathcal{T}f|^p(x) w(x) dx \leq C\int_{\R}
|\mathcal{M}_{\mathcal{T}}f|^p(x)w(x)dx,
\end{align}
for some $0< p < \infty$ and some $0 \leq w \in L^1_{\text{loc}}(\R)$.

 The maximal operator $\mathcal{M}_{\mathcal{T}}$ is related to the operator $\mathcal{T}$ which is normally easier to deal with. In general, $\mathcal{M}_{\mathcal{T}}$ is strongly related to the kernel of $\mathcal{T}$.


The classical result of Coifman in \cite{Co} is, let $T$ be a
Calder\'{o}n-Zygmund operator, then $T$ is controlled by $M$, the
Hardy-Littlewood maximal operator.
In other words, for all
$0<p<\infty$ and $w\in A_{\infty}$,
$$\int_{\R} |Tf|^p(x) w(x) dx \leq C\int_{\R} (Mf)^p(x)w(x)dx.$$

Later in \cite{RuRTo}, Rubio de Francia, Ruiz and Torrea studied operators with less regularity in the kernel. They proved that for certain operators,  (\ref{Coifineq}) holds with $\mathcal{M}_{\mathcal{T}}=M_r f=M(|f|^r)^{1/r}$, for some $1\leq r<\infty$. The value of the exponent $r$ is determined by the smoothness of the kernel, namely, the kernel satisfies an $L^{r'}$-H\"ormander condition (see the precise definition below).
In \cite{MPTG}, Martell, P\'erez and Trujillo-Gonz\'alez proved that this control is sharp in the sense that one cannot write a pointwise smaller operator $M_s$ with $s<r$.
This yields, that for operators satisfying only the classical H\"ormander condition, $H_1$, the inequality (\ref{Coifineq}) does not hold for any $M_r$, $1\leq r <\infty$.\\

More recently, in \cite{LRTo}, Lorente, Riveros and de la Torre defined a $L^{\A}$-H\"ormander condition where $\A$ is a Young function.
 If $T$ is an operator such that satisfies this condition, then (\ref{Coifineq}) holds  for $M_{\overline{\A}}$, the maximal operator
 associated to the Young function $\overline{\A}$.

As a consequence of the Coifman inequalities, one can prove weighted modular end-point estimates. In \cite{LMRT}, Lorente, Martell, Riveros and de la Torre proved the following:
if $\overline{\A}$ is submultiplicative and $\lambda >0$, then
$$w\{x \in \R : \vert T f(x) \vert> \lambda\} \leq c \int_{\R}\overline{\A}\left( \frac{|f(x)|}{\lambda}\right)Mw(x) dx.
$$

An example of this type of operator is the square operator $S$ (see the precise definition below), by the results in \cite{LRTo} the following inequality holds,
$$\int_{\R} |Sf|^p(x) w(x) dx \leq C\int_{\R} (M^{3}f)^p(x)w(x)dx,$$
 for all $0<p<\infty$ and $w\in A_{\infty}$.
In \cite{LRTo2} it was proved that the last inequality is not sharp in the sense that  it  can be replaced $M^{3}$ by $M^{2}$. \\

In this paper we define a new H\"ormander condition in the case of vector-valued kernels, weaker than the $L^{\A}$-H\"ormander condition defined in \cite{LMRT}.  We obtain inequality (\ref{Coifineq}) improving results, for vector-valued operators, obtain in  \cite{LMRT}. The applications of this results with the new condition are generalizations of ones for the square operator  obtained  \cite{LRTo2}.  In these applications, the maximal operators are of the form $M_{L\log L^{\beta}}$, with some $\beta \geq 0$.  For instance, we obtain for all $0<p<\infty$ and $w\in A_{\infty}$,

$$\int_{\mathbb{R}} |S_{X}f|^p(x) w(x) dx \leq C\int_{\mathbb{R}} (M^{2}f)^p(x)w(x)dx,$$
where $X$ is an appropriate Banach space.  If $X=l^2$, $S_{X}=S$ the square operator, and in this case  we obtain the same results as in  \cite{LRTo2}.\\

In \cite{BLR}, Bernardis, Lorente and Riveros defined
$L^{\A,\alpha}$-H\"ormander conditions for fractional integral operator. 
The authors obtain the inequality (\ref{Coifineq}) with
$M_{\overline{\A},\alpha}$, the
fractional maximal operator associated to $\overline{\A}$. In this paper, we also
give a  weaker condition for vector-valued kernels than
$L^{\A,\alpha}$-H\"ormander
condition  and obtain similar  kind of results and applications.\\

 The plan  of this paper is as follows. The next section contains some definitions and well known results.
  Later, in section 3, we introduced our condition and the main results. The applications are presents section 4. The proofs of the general results are in sections 5. Finally in the last section we present the H\"ormander condition and the results for vector-valued fractional operators.

\section{Preliminaries}
In this section we present some notions needed to understand the
main results and the applications.
 First we define the space in which we are going
to work.

Let us consider the Banach spaces $(X,\| \cdot\|_{X})$ where $X=\mathbb{R}^{\Z}$ and the norm in this space is monotone, i.e.
\begin{align}\label{definorma} \|\{a_n\}\|_{X} &\leq \|\{b_n\}\|_{X} \quad \quad \text{ if } |a_n| \leq |b_n| \; \text{ for all } \, n \in \Z.
\end{align}
Observe that $\| \left\{ a_n \right\} \|_{X} = \|
\left\{|a_n|\right\}\|_{X} \quad \text{ for all } \; \{a_n\} \in X$.

\begin{obs} Some examples of this Banach spaces are the $l^p(\Z)$ spaces, $1\leq p<\infty$, and the space where the norm is associated to some Young function.
Observe that not all Banach spaces satisfies the condition
(\ref{definorma}), for example, consider $X = \mathbb{R}^{\Z}$ with the norm
$$ \|\{x_n\}\|_{X}:=\left((x_1-x_2)^2 + \sum_{n\neq 1}
x_n^2\right)^{1/2}.$$ Let
 $(...,0,x_1,x_2,0,...)=(...,0,1,3,0,...)$ and $(...,0,y_1,y_2,0,...)=(...,0,2,3,0,...)$. Observe that $|x_n| \leq |y_n|$ for all  $n \in \Z$, and $\|\{x_n\}\|_{X}=\sqrt{13}$ y $\|\{y_n\}\|_{X}=\sqrt{10}$. Hence, the norm is not monotone.
\end{obs}

\begin{obs}
If $X$ is a Banach lattice, the norm is monotone by definition.
\end{obs}



Now, we define the notion of Young function, maximal operators
related to Young function and generalized H\"ormander condition. For more details see \cite{ON}.

A function $\A : [0,\infty) \rightarrow [0,\infty)$ is said to be a Young function if $\A$ is continuous, convex, no decreasing and satisfies $\A(0)=0$ and $\displaystyle \lim_{t \rightarrow \infty} \A(t)= \infty$.



The average of the Luxemburg norm of a function $f$ induced by a Young function $\A$ in the ball $B$ is defined by
\begin{align*} \|f\|_{\A,B}:= \inf \left\{\lambda >0:\, \frac1{|B|} \int_{B} \A\left(\frac{|f|}{\lambda}\right) \leq 1  \right\}.\end{align*}

Observe that if $\A(t)=t^r$, $r\geq 1$,  $\|f\|_{\A,B}=\left( \frac1{|B|}\int_{B} |f|^r\right)^{1/r}$.

Each Young function $\A$ has an associated complementary Young function $\overline{\A}$ satisfying
the generalized H\"older inequality
 \begin{align*} \frac1{|B|} \int_{B} |fg| \leq 2\|f\|_{\A,B}\|g\|_{\overline{\A},B}. \end{align*}

If $\A,\B,\mathcal{C}$ are Young functions satisfying $\A^{-1}(t)\B^{-1}(t)\mathcal{C}^{-1}(t)\leq t$, for all $t \geq 1$, then \begin{align*} \|fgh\|_{L^1,B} \leq c \|f\|_{\A,B}\|g\|_{\B,B}\|h\|_{\mathcal{C},B}. \end{align*}

Given $f \in L^1_{\text{loc}}(\R)$, the  maximal operator associated to the Young function  $\A$ is defined as
\begin{align*} M_{\A}f(x) := {\underset{B \ni x}{\sup}} \|f\|_{\A,B}. \end{align*}

For example,
if $\beta \geq 0$ and $r \geq 1$, $\A(t)= t^r(1 + \log(t))^{\beta}$ is Young function then $M_{\A}=M_{L^r(\log L)^{\beta}}$. If $\beta =0$, $\A(t)=t^r$ then $M_{\A}=M_r$, where $M_r f= M(f^r)^{1/r}$. If $r=1$ and $\beta=k\in \N$, $M_{\A}=M_{L(\log L)^{k}}\approx M^{k+1}$, where $M^k$ is the $k$-iterated of M, maximal of Hardy-Littlewood.

\begin{obs} Let us observe that when $\mathcal{D}(t)=t$, which gives $L^1$, then  $\overline{\mathcal{D}}(t)=0$ if $t\leq1$ and $\overline{\mathcal{D}}(t)=\infty$ otherwise. Observe that $\overline{\mathcal{D}}$ is not a Young function but one has $L^{\overline{\mathcal{D}}}=L^{\infty}$. Besides, the inverse is $\overline{\mathcal{D}}^{-1}\equiv 1$ and the generalized H\"older inequality make sense if one of the three function is $\overline{\mathcal{D}}$.
\end{obs}

Once the Luxemburg average has been defined, we can introduce the notion of the ge\-ne\-ra\-li\-zed H\"ormander condition, for this we need to introduce
 some notation: $|x|\sim s$ means $s < |x|\leq 2s$ and given a Young function $\A$, we write,
$$\|f\|_{ \A,|x|\sim s}=\|f \chi_{|x|\sim s}\|_{\A,B(0,2s)}.$$
In \cite{LMRT} and \cite{LRTo}  were introduced the following classes,

\begin{mydefA}\label{HA}
Let $K$  be a vector-valued function, $\A$ be a Young function and $k \in \N\cup \{0\}$, then
$K$ satisfies the $L^{\A,X,k}$-H\"ormander condition ($K \in H_{\A,X,k}$), if there exist $c_{\A}>1$ and $C_{\A}>0$ such that for all $x$ and $R>c_{\A}|x|$:
 \begin{align*}\sum_{m=1}^{\infty} (2^mR)^n m^k \big\| \| K(\cdot - x) - K(\cdot)\|_{X}\big\|_{\A,|y|\sim2^mR} \leq C_{\A}.\end{align*}
We say that $K \in H_{\infty,k}$ if $K$ satisfies the previous condition with $\|\cdot\|_{L^{\infty},|x|\sim 2^mR}$ in place of $\|\cdot\|_{\A,|x|\sim 2^mR}$.

If $k=0$, we denote $H_{\A,X}=H_{\A,X,0}$ y $H_{\infty,X}=H_{\infty,X,0}$.
\end{mydefA}

\begin{obs}\label{relHA} There exists  a relation between the H\"ormander classes, $H_{\A,X,k}$.  
\begin{enumerate}
    \item $H_{\infty,X,k} \subset H_{\A,X,k}\subset H_{\A,X,k-1} \subset \cdots \subset H_{\A,X,0}=H_{\A,X} \subset H_{1,X}$,  for $ k\in \N$.
    \item If  $\A$ and $\B$ are Young functions such that  $\A(t) \leq c\B(t)$ for $t > t_0$, some $t_0>0$, then:
\begin{align*} H_{\infty,X,k} \subset H_{\B,X,k} \subset H_{\A,X,k} \subset H_{1,X,k} \subset H_{1,X}.
\end{align*}
    \item In the particular case of $\A(t)=t^r$, $1\leq r < \infty$ denoting $H_{r,X} = H_{\A,X}$, it follows that, \begin{align*} H_{\infty,X,k} \subset H_{r_2,X,k} \subset H_{r_1,X,k} \subset H_{1,X,k} \subset H_{1,X},  \quad \text{ for all } \,\, 1<r_1 < r_2 < \infty.\end{align*}
\end{enumerate}
\end{obs}
\vspace*{0.5cm}

Next, we define the notions of  singular integral  operator and its commutator in the vector-valued sense.
\begin{mydef} Considerate a  vector-valued function $K$, $K(y)= \{ K_{l}(y)\}_{l\in \Z}$, with
\\$K_l \in L_{\text{loc}}^1(\R \diagdown \{0\})$. Let,
 \begin{align*} Tf(x)&:= v.p. \;\int_{\R} K(x - y)f(y)dy = \{(K_l \ast f )(x)\}_{l \in \Z} & \\&= \left\{ v.p. \; \int_{\R} K_l (x-y) f(y) dy\right\}_{l \in \Z}.
\end{align*}The operator $T$ will be a singular integral operator  if it is strong $(p_0,p_0)$, for some $p_0 >1$, and the kernel $K=\{K_l\}_{l\in \Z} \in H_{1,X}$.
\end{mydef}
\begin{obs}\label{Bochner}
The operator $T$ is strong $(p_0,p_0)$ in the sense of Bochner-Lebesgue spaces. Given a $X$ Banach space, $L^p_X(\R)$ is called Bochner-Lebesgue spaces with the norm\\ $\left(\int_{\R}\|f(x)\|^p_Xdx\right)^{1/p}$.
\end{obs}
\begin{obs}
Since $K=\{K_l\}_{l\in \Z} \in H_{1,X}$, then $T$ is of weak type (1,1).
Thus, using the fact that  $T$ is of strong type $(p_0,p_0)$ by
interpolation and duality, $T$ is of strong type $(p,p) \; \forall
\, 1 < p < \infty$.

Moreover, since $T$ is of  weak type (1,1), $T$ satisfies Kolmog\"orov's inequality

$$\left( \dfrac{1}{\vert B \vert} \int_{B} \Vert Tf \Vert_{\mathbf{x}}^{\varepsilon} \right)^{\frac{1}{\varepsilon}} \leq c \; \dfrac{1}{\vert \widehat{B} \vert} \int_{\widehat{B}} \vert f \vert,$$
where $0<\epsilon <1$ and supp$(f)=\widehat{B} \subset B$.
\end{obs}

Let us recall the $BMO$ space and the sharp maximal function. If $f \in L_{loc}^1(\R)$ define
$$ M^{\#}f(x)=  {\underset{B\ni x}{\sup}}\frac1{|B|}\int_{B}\left|f-\frac1{|B|}\int_{B}f\right|.$$

A locally integrable function $f$ has bounded mean oscillation ($f \in BMO$) if $ M^{\#}f \in L^{\infty}$ and the norm
$\|f\|_{BMO}=\|M^{\#}f\|_{\infty}$

Observed that the BMO norm is equivalent to
$$||f||_{BMO}=\|M^{\#}f\|_{\infty} \sim {\underset{B}{\sup}} {\underset{a \in \C}{\inf}} \frac1{|B|} \int_{B} |f(x)-a| dx. $$

\begin{obs} Some properties of $BMO$ are the following.

Given $b\in BMO$, a ball $B$, $k \in \N\cup \{0\}$,
$\A(t)=\exp(t^{1/k})$ and $q>0$, by John-Nirenberg's Theorem we have
\begin{equation}\label{JN1}
\|(b-b_B)^k\|_{L^q,B}  \leq   \|(b-b_B)^k \|_{\A,B} = \|b-b_B\|^k_{\exp L,B} \leq  C \|b\|_{BMO}.
\end{equation}

On the other hand, for any $j\in \N$ and $b \in BMO$, we have
\begin{equation}\label{JN2}
|b_{B}-b_{2^jB}| \leq \sum_{m=1}^j |b_{2^{m-1}B}-b_{2^mB}| \leq 2^n \sum_{m=1}^j \|b-b_{2^mB}\|_{L^1, 2^mB}  \leq 2^n  j \|b\|_{BMO}.
\end{equation}
\end{obs}

\begin{mydef}  Given $T$ a singular integral operator and $b \in BMO$, it is define the $k$-th order commutator of  $T$, $k \in \N \cup \{0\}$, by:
\begin{align*} T_b^k f(x) & := v.p.
 \int_{\R} (b(x) - b(y))^k K(x-y)f(y) dy  \\ &= \left\{ v.p. \int_{\R}(b(x) - b(y))^k K_l(x-y)f(y) dy \right\}_{l \in \Z}.\end{align*}
Note that for $k = 0$, $T_b^k = T$ and observe that $T_b^k=[b, T_b^{k-1}]$, $k\in \N$.

\end{mydef}

\begin{obs} $T_b^kf(x)=[b, T_b^{k-1}](f)(x) := b(x)T_b^{k-1}(f)(x) - T_b^{k-1}(bf)(x)$.
\end{obs}

We will consider weights in the Muckenhoupt classes $A_p$, $1\leq p \leq \infty$. Let $w$ be a non-negative locally integrable function. We say that $w \in A_p$ if there exists $C_p<\infty$ such that for any ball $B \subset \R$,
$$  \left(\frac1{|B|}\int_B w\right)\left(\frac1{|B|}\int_B w^{-\frac{1}{p-1}}\right)^{p-1}<C_p,$$
when $1<p<\infty$, and for p=1,
$$  Mw(x)\leq C_1 w(x), \qquad \text{ for a.e } x\in \R.$$

Finally we set $A_{\infty}=\cup_{1<p}A_p$. It is well known that the Muckenhoupt classes characterize the boundedness of the Hardy-Littlewood maximal function on weighted $L^p$-Lebesgue spaces. Namely, $w \in A_p$, $1<p\leq \infty$, if and only if $M$ is bounded on $L^p(w)$; and $w\in A_1$ if and only if $M$ maps $L^1(w)$ into $L^{1,\infty}(w)$.

In \cite{LRTo} and \cite{LMRT}, the following results were proved,

 \begin{teoA}
\cite{LRTo}  \label{hxa} Let
$K$ be a vector-valued function that satisfies the $L^{\A,
X}$-H\"ormander condition and let $T$ be the operator associated  to
$K$. Suppose $T$ is bounded in some $L^{p_0}$, $1<p_0<\infty$. Then, for
any $0<p<\infty$ and $w\in A_\infty,$ there exists $C$ such that
$$
\int_{\R}\Vert Tf\Vert_\mathbf{x}^pw\leq C\int_{\R}(M_{\overline
{\A}}f)^p w,
$$
for any $f\in C_c^{\infty}$ and whenever the left-hand side is finite.
\end{teoA}

For commutators of the operator $T$, there is the following result:

\begin{teoA}
 {\rm \cite{LMRT}}\label{teo1} Let $b \in BMO$ and $k \in \N \cup \{0\}$. Let $\A$, $\B$ Young functions such that $\overline{\A}^{-1}(t)\B^{-1}(t)\overline{\mathcal{C}}_k^{-1}(t) \leq t$, with $\overline{\mathcal{C}}_{k}(t)=\exp(t^{1/k})$ for $t \geq 1$ if $k\in \N$ and  $\overline{\mathcal{C}}_k^{-1}\equiv 1$ if $k=0$.  If T is a singular integral operator with kernel  $K \in H_{\B,X} \cap H_{\A,X,k}$, then for any $ 0 < p < \infty$ and $w \in A_{\infty}$,
\begin{align*} \int_{\R} \| T_b^kf\|^p_X w \leq C \int_{\R} (M_{\overline{\A}}f)^p w, \quad f \in L_c^{\infty},\end{align*}
whenever the left-hand side is finite.
Furthermore, if $\overline{\A}$ is sub-multiplicative, then for all  $w \in A_{\infty}$ and $\lambda >0$,
$$
w\{x \in \R : \vert T_b^k f(x) \vert> \lambda\} \leq c \int_{\R}\overline{\A}\left( \frac{\|b\|_{BMO}^{k}|f(x)|}{\lambda}\right)Mw(x) dx.
$$
\end{teoA}


\section{Main results}

In this section we will state a new condition weaker than the
generalized H\"ormander condition (Definition \ref{HA}). The previous
Theorems \ref{hxa} and \ref{teo1}  still remain true using this new condition. 

\begin{mydef}\label{daga} Let $K$ be a vector-valued function, $\A$ be a Young function and $k \in \N\cup \{0\}$. The function
 $K$ satisfies the $L^{\A,X,k}_{\dagger}$-H\"ormander condition ($K \in H^{\dagger}_{\A,X,k}$), if there exist $c_{\A}>1$ and $C_{\A}>0$ such that for all $x$ and $R>c_{\A}|x|$,
 \begin{align*}\Biggl\| \biggl\{ \sum_{m=1}^{\infty} (2^mR)^n m^k  \| K_l(\cdot - x) - K_l(\cdot)\|_{\A,|y|\sim2^mR}\biggr\}_{l\in \Z} \Biggr\|_{X} \leq C_{\A}.\end{align*}
We say that $K \in H^{\dagger}_{\infty,k}$ if $K$ satisfies the previous condition with $\|\cdot\|_{L^{\infty},|x|\sim 2^mR}$ in place of $\|\cdot\|_{\A,|x|\sim 2^mR}$.

If $k=0$, we denote $H^{\dagger}_{\A,X}=H^{\dagger}_{\A,X,0}$ y $H^{\dagger}_{\infty,X}=H^{\dagger}_{\infty,X,0}$.\\
\end{mydef}

\begin{obs} The classes $H_{\A,X,k}^{\dagger}$ satisfies  the
same inclusion of the classes $H_{\A,X,k}$, see remark \ref{relHA}.
And the relation between this classes is the following,
$$H_{\A,X,k} \subsetneq H_{\A,X,k}^{\dagger}.$$
\end{obs}

In section 4, we give an explicit example of a kernel $K$ such that
$K\in H_{\A,X,k}^{\dagger}$ and $K \not \in H_{\A,X,k}$, (see
Proposition \ref{prop2} and Corollary \ref{prop4}).

Using Definition \ref{daga}, the previous theorem are written as follows, for the case $k=0$,

\begin{teo}\label{TeoImp} Let $T$ be a  vector-valued singular integral operator with kernel $K \in H_{\A,X}^{\dagger}$. Then, for any $0 < p < \infty$ and $w \in A_{\infty}$, there exist $C$ such that
\begin{align*} \int_{\R} \| Tf\|^p_X w \leq C \int_{\R} (M_{\overline{\A}}f)^p w, \quad f \in L_c^{\infty}(\R),\end{align*}
whenever the left-hand side is finite.
\end{teo}

And for the case $k \in \N$,

\begin{teo}\label{TeoP}
 Let $b \in BMO$ and $k \in \N $. Let $\A$, $\B$ be Young functions such that \\ $\overline{\A}^{-1}(t) \B^{-1}(t) \overline{\mathcal{C}}_k^{-1}(t) \leq t$, with $\overline{\mathcal{C}}_{k}(t)=\exp(t^{1/k})$ for $t \geq 1$. If T is a vector-valued singular integral operator with  kernel $K \in H_{\B,X}^{\dagger} \cap H_{\A,X,k}^{\dagger}$, then for any
 $ 0 < p < \infty$ and $w \in A_{\infty}$, there exists $C$ such that
\begin{align*} \int_{\R} \| T_b^kf\|^p_X w \leq C \int_{\R} (M_{\overline{\A}}f)^p w, \quad f \in L_c^{\infty}(\R),\end{align*}
whenever the left-hand side is finite.

Furthermore, if $\overline{\A}$ is sub-multiplicative, then for all  $w \in A_{\infty}$ and $\lambda >0$,
$$
w\{x \in \R : \vert T_b^k f(x) \vert> \lambda\} \leq c \int_{\R}\overline{\A}\left( \frac{\|b\|_{BMO}^{k}|f(x)|}{\lambda}\right)Mw(x) dx.
$$

\end{teo}

 \begin{obs} These theorems are more general than Theorem \ref{hxa} and \ref{teo1}, since there exists a singular integral operator whose kernel $K \in H_{\A,X,k}^{\dagger}$ and $K \not \in H_{\A,X,k}$ for some appropriate Young function $\A$.
\end{obs}


Let $\A(t)=\exp(t^{\frac1{1+k}})-1$ and $\overline{\mathcal{C}}_{k}(t)=\exp(t^{1/k})$. If $\B(t)=\exp(t)-1$ then
 $\overline{\A}^{-1}(t) \B^{-1}(t) \overline{\mathcal{C}}_k^{-1}(t) \\ \leq t$. Thus, if $K \in H_{\A,X,k}^{\dagger}$ then $K \in H_{\B,X}^{\dagger}$.
  In this case Theorems \ref{TeoImp} and  \ref{TeoP} can be written as follows:

\begin{teo}\label{Teo37}
 Let $b \in BMO$ and $k \in \N \cup \{0\}$. Let $\A(t)=\exp(t^{\frac1{1+k}})-1$. If T is a vector-valued singular integral operator with  kernel
  $K \in H_{\A,X,k}^{\dagger}$, then for any $ 0 < p < \infty$ and $w \in A_{\infty}$, there exists $C$ such that
\begin{align*} \int_{\R} \| T_b^kf\|^p_X w \leq C \int_{\R} (M_{\overline{\A}}f)^p w \leq C \int_{\R} (M^{k+2}f)^p w, \quad f \in L_c^{\infty}(\R),\end{align*}
whenever the left-hand side is finite.

Furthermore, for all  $w \in A_{\infty}$ and $\lambda >0$,
$$
w\{x \in \R : \vert T_b^k f(x) \vert> \lambda\} \leq c \int_{\R}\overline{\A}\left( \frac{\|b\|_{BMO}^{k}|f(x)|}{\lambda}\right)Mw(x) dx,
$$
where $\overline{\A}(t)=t(1 + log(t))^{k+1}$.
\end{teo}

\section{Applications and generalization}

Now, we define the vector-valued singular integral operator, $\tilde{T}$, and its commutator, that will be an example of our results.

\begin{mydef}\label{defTtilde}Let $f$ be a locally integrable function in $\mathbb{R}$. Let $\tilde{T}$ be defined  as: \begin{align*}\tilde{T}f(x) &:= \left\{ \int_{\mathbb{R}} \left( \frac1{2^{l+1}}\chi_{(-2^l,2^l)}(x-y) - \frac1{2^l} \chi_{(-2^{l-1},2^{l-1})}(x-y) \right)f(y)dy\right\}_{l\in\Z}\\&= \int_{\mathbb{R}} K(x-y)f(y)dy,\end{align*} where $K$ is \begin{align*} K(z) =\{ K_l(z)\}_{l\in\Z} = \left\{ \frac1{2^{l+1}}\chi_{(-2^l,2^l)}(z) - \frac1{2^l} \chi_{(-2^{l-1},2^{l-1})}(z) \right\}_{l\in\Z}.\end{align*}
\end{mydef}

For this operator $\tilde{T}$, the Banach space $(X, \|\cdot\|_{X})$ 
 will be $(l^2(\Z), \|\cdot\|_{l^2})$.

\begin{mydef}Let $f$ be a measurable function in $\mathbb{R}$, $k \in \N \cup \{0\}$ and $b \in BMO$. The $k$-th order commutator is defined as,
 \begin{align*} S_b^k f(x) := \|\tilde{T}_b^k f(x)\|_{l^2},\end{align*}
where $\tilde{T}_b^k$ is the $k$-th order commutator of $\tilde{T}$.
The $S_{b}^k$ is called the $k$-th commutator of the square operator.
\end{mydef}



In \cite{ToTo} and \cite{LMRT}, the authors studied the kernel of the square operator for the one-sided case, the results for the two-sided case are the following and the proof are analogous to the one-sided case.

\begin{propA}\label{defK}\cite{ToTo} Let $x_0 \in \mathbb{R}\,$ and $ i < j, \, i,j \in \Z$. Let $x,y \in \mathbb{R}$ such that $|x-x_0|<2^i$, $y \in (x_0 - 2^{j+1},x_0 -2^j)$ or $y \in (x_0 + 2^{j},x_0 +2^{j+1})$. Then \begin{align*} |K_l(y-x)-K_l(y-x_0)|=
\begin{cases}
\frac1{2^{j+1}} \chi_{(x-2^j,x_0 - 2^j)\,\cup\,(x_0 + 2^j,x + 2^j)}(y) & \text{if } l=j,\\
\frac1{2^{j+2}} \chi_{(x_0-2^{j+1},x - 2^{j+1})\,\cup\,(x + 2^{j+1},x_0 + 2^{j+1})}(y) \\+ \frac1{2^{j+1}} \chi_{(x-2^j,x_0 - 2^j)\,\cup\,(x_0 + 2^j,x + 2^j)}(y) & \text{if } l=j+1,\\
\frac1{2^{j+2}} \chi_{(x_0-2^{j+1},x - 2^{j+1})\,\cup\,(x + 2^{j+1},x_0 + 2^{j+1})}(y)& \text{if } l=j+2, \\
0 & \text{if } l \not\in \{j,j+1,+2\}.
\end{cases} \end{align*}
\end{propA}

In \cite{LRTo}, using Proposition \ref{defK} the authors proved the
following results
\begin{propA}\label{prop1}\cite{LRTo} The kernel $K \not \in H_{\infty,\, l^2}$.
\end{propA}

\begin{obs}  As $K \not \in H_{\infty,\,l^2,k}$ we can not use  Theorem \ref{TeoP} to conclude
\begin{align*} \int_{\mathbb{R}} |S_b^kf(x)|^pw(x)dx = \int_{\mathbb{R}} \|\tilde{T}_b^kf(x)\|_{l^2}^pw(x)dx  \leq C \int_{\mathbb{R}} |M^{k+1}f(x)|^pw(x)dx. \end{align*}
This inequality is still an  open problem.
\end{obs}

\begin{propA}\label{prop2}\cite{LMRT} Let $\A_{\varepsilon}(t)=\exp(t^{\frac1{1+k+\varepsilon}})-1$, $\varepsilon \geq 0$ and $k\in \N \cup \{0\}$. Then, $K \in H_{\A_{\varepsilon},l^2,k}$   for all $\varepsilon > 0$, and $K \not \in H_{\A_0,l^2,k}$.
\end{propA}

 In \cite{LRTo} and \cite{LMRT}, as an application of Theorems
\ref{hxa} and \ref{teo1} the authors  obtained the following result



\begin{teoA} {\rm \cite{LRTo,LMRT}}\label{corol1} Let $b \in BMO$ and $k \in \N \cup \{0\}$. Let $S_b^k$ be the $k$-th order commutator of the square operator. Then
for any $0<p<\infty$ and $w \in A_{\infty}$, there exist $C$ such that
\begin{align*} \int_{\mathbb{R}} (S_b^kf(x))^pw(x)dx =  \int_{\mathbb{R}} (\|\tilde{T}_b^kf(x)\|_{l^2})^pw(x)dx \leq C \int_{\mathbb{R}} (M^{k+3}f(x))^pw(x)dx,\end{align*}
whenever the left-hand side is finite.
\end{teoA}

For the case of the kernel of the the square operator we obtain,

\begin{prop}\label{prop3}Let $k \in \N \cup \{0\}$ and $\A$ be a  Young function. Then,
$$K \in H_{\A,\,l^2,k}^{\dagger} \Leftrightarrow \left\| \left\{ \frac{m^k}{\A^{-1}(2^{m}8)}\right\}_{m\in\Z}\right\|_{l^2} < \infty. $$
\end{prop}

\begin{corol}\label{prop4}Let $\A(t)=\exp(t^{\frac1{1+k}})-1$. Then the kernel $K \in H_{\A,\,l^2,k}^{\dagger}$ for any $k \in \N \cup \{0\}$.
\end{corol}

Corollary \ref{prop4} tell us that the kernel of the square operator satisfies the hypothesis of Theorems \ref{TeoImp} and \ref{TeoP}  (see Theorem \ref{Teo37} 
) and we obtain a new proof of the following result,


\begin{teoA} {\rm \cite{LRTo2}}\label{corol2} Let $b \in BMO$ and  $k \in \N \cup \{0\}$. Let $S_b^k$ be the $k$-th order commutator of the square operator. Then, for any $0<p<\infty$ and $w \in A_{\infty}$, there exists $C$ such that
\begin{align*} \int_{\mathbb{R}} (S_b^kf(x))^pw(x)dx  \leq C \int_{\mathbb{R}} (M^{k+2}f(x))^pw(x)dx,\end{align*}
 whenever the left-hand side is finite.
\end{teoA}


\subsection{Generalization of square operator}

In this subsection, we will build a family of operators and we will prove that they satisfy  Proposition \ref{prop3}. This operators are a ge\-ne\-ra\-li\-za\-tion of the square operator.

Let $X$ be a Banach space with a monotone norm, see (\ref{definorma}). We define  $S_{X}f(x):=||\tilde{T}f(x)||_{X}$, where $\tilde{T}$ was  define in  Definition \ref{defTtilde}. Observe that if $X=l^2$ then $S_{X}=S$, the square operator.

We can generalize Proposition \ref{prop3} and Corollary \ref{prop4}, replacing the $l^2$-norm by $X$-norm.
\\
In this context  Proposition \ref{prop3} affirm,  for all $k\geq 0$, and  $\A$ be a  Young function,
\begin{align}\label{condp3}
K \in H_{\A,X,k}^{\dagger} \Leftrightarrow \left\| \left\{ \frac{m^k}{\A^{-1}(2^{m}8)}\right\}_{m\in\Z}\right\|_{X} < \infty.
\end{align}
Also Corollary \ref{prop4}  can be rewritten in this way,  if $\A(t)=\exp(t^{\frac1{1+k}})-1$  and $k\in \N \, \cup \, \{0\}$, then  $ K \in H_{\A,X,k}^{\dagger}$.
\\

 Observe that if $k \in \N \cup \{0\}$ and $\A(t)=\exp(t^{\frac1{1+k}})-1$, by
 Proposition  \ref{prop3}, we have

\begin{align}\label{condAE}
K \in H_{\A,X,k}^{\dagger} \Leftrightarrow  \left\Vert \left\lbrace  \frac1{m} \right\rbrace_{m \in (\Z-\{0\})} \right\Vert_{X}=C_{\A,X}<\infty.
\end{align}

 Applying Theorem \ref{Teo37}, we obtain

\begin{align*}
\int_{\R} \vert S_{X, b}^k f(x) \vert^{p} w(x) dx &= \int_{\R} \Vert \tilde T_b^kf(x) \Vert_{X}^{p} \, w(x) dx\\
 &\leqslant c \int_{\R} \left(M_{\overline{\A}}f (x) \right)^{p} w(x)dx \leqslant c \int_{\R} \left(M^{k+2}f (x) \right)^{p} w(x)dx,
\end{align*}

whenever the left-hand side is finite.

\begin{obs}
Examples of the Banach spaces $X$ are the $l^p$ spaces with $1\leq p
\leq \infty$. Observe that for $p=2$,  condition \ref{condAE} holds,
but for $p=1$ is easy to see that this condition does not hold. One open
question is: there exists a Young function $\A$ such that the
condition \ref{condp3} is finite for $X=l^1$? For example, there exists a
Young function such that the condition \ref{condAE} is replace by
$\left\| \left\{  \frac1{m^2} \right\}_{m \in (\Z-\{0\})}
\right\|_{l^1}$?
\end{obs}

A interesting example is, given a Young function $\mathcal{E}$,  we denote  $X_{\mathcal{E}}=({\mathbb{R}}^{\Z},\|\cdot\|_{\mathcal{E}})$, the Banach space with $$\displaystyle \|\{a_n\}\|_{\mathcal{E}}=\inf \left\{\lambda >0:\, \sum_{n \in \Z} \mathcal{E}\left(\frac{|a_n|}{\lambda}\right) \leq 1  \right\}.$$

Now we give an example of a family of Young functions, $\mathcal{E}$,
 for which condition (\ref{condAE}) holds. Let us consider,  for $t \geq 0$, the Young function ${\mathcal{E}}(t)=t^{r} \left(\log(t+1)\right)^{\beta}$, where $\beta \geq 0$ and $r \geq 1$.

Observe  that to  prove this assertion is equivalent to  prove that  there exist $0<\lambda < \infty$ such that  $\lambda \in G$, where $G$ is defined as
$$G:= \left\{ \lambda>0 \; : \; \sum_{m \in \Z-\{0\}} \mathcal{E}\left(\frac{1/m}{\lambda}\right) \leq 1 \right\}.$$

\vspace{0,5em}

Let $\lambda >1$,

$$
\begin{aligned}
\sum_{m \in \Z-\{0\}} \mathcal{E}\left(\frac{1/m}{\lambda}\right)  &= \sum_{m \in (\Z- \{0\})} \log\left(\frac{1}{|m| \lambda}+1\right)^{\beta} \frac{1}{(\lambda m)^{r}}\\
& \leq \log\left(\frac{1}{ \lambda}+1\right)^{\beta} \frac{1}{\lambda^{r}} \sum_{m  \in (\Z- \{0\})} \frac{1}{m^{r}}\\
& \leq \log(2)^{\beta} \frac{1}{\lambda^{r}} c.
\end{aligned}
$$

\vspace{1em}

Observe that $\log(2)^{\beta} \frac{1}{\lambda^{r}} c \leq 1$ if and only if $c \log(2)^{\beta} \leq \lambda^{r}$. In particular, $\lambda_0:=\left(\log(2)^{\beta} C +1\right)^{1/r}$ satisfies this inequality, i.e., $\lambda_0 \in G$. Thus, we have  (\ref{condAE}) is true.

\subsection{Proof of Proposition \ref{prop3} and Corollary \ref{prop4}}

In this subsection we proceed to study the applications. Let $K$ be the kernel of the square operator, defined above.

\begin{mydef*}
We define the sets,
\begin{align*} -F_m^{-} := (x-2^{m+i},-2^{m+i})  && F_m^{-} := (x+2^{m+i},2^{m+i}) \\
-F_m^{+}:=(-2^{m+i},x-2^{m+i}) && F_m^{+}:=(2^{m+i},x+2^{m+i})
 \end{align*}
\begin{align*} -F_m :=
\begin{cases}
-F_m^{-} & \text{si } x<0\\
-F_m^{+} & \text{si } x>0
\end{cases} && F_m :=
\begin{cases}
F_m^{-} & \text{si } x<0\\
F_m^{+} & \text{si } x>0
\end{cases} \end{align*}
\end{mydef*}
Observed that if $|x|<2^i$, $[-F_m \cup F_m]\cap[-F_{m-1} \cup F_{m-1}] = \emptyset$, for all $ m\in \Z$.

\begin{proof}[Proof of Proposition \ref{prop3}]
Recall $K \in H_{\A,\,l^2,k}^{\dagger}$ if  there exist $c_{\A}>1$ and $C_{\A}>0$ such that for each $x$ and $R > c_{\A} |x|$,\\
\begin{equation}\label{Hdag} \left\|\left\{\sum_{m=1}^{\infty} (2^mR)^n m^k \| (K_l(\cdot - x) - K_l( \cdot)) \chi_{|y|\sim2^mR}(\cdot)\|_{\A,B(0,2^{m+1}R)}\right\}_{l\in\Z}\right\|_X \leq C_{\A}.\\\end{equation}
\\

Let us prove, \begin{center}$\left\| \left\{ \frac{m^k}{\A^{-1}({2^{m}}8)}\right\}_{m\in\Z}\right\|_{l^2} < \infty $ $\Rightarrow$  $K \in H_{\A,l^2,k}^{\dagger}$. \end{center}

If $x=0$, $(K_l( \cdot - x) - K_l(\cdot))=0$ for all $\; l \in \Z$, then the condition (\ref{Hdag}) is trivial.
Let $x \neq 0$.
Let $R=2^i$, $i \in \Z$, $x$ such that $|x|<2^i$, $I_m:=(-2^{m+i},2^{m+i})$ and $-F_m$ y $F_m$ as above. For $l \in \Z$, using Proposition \ref{defK} we obtain
\begin{align*}
\sum_{m=1}^{\infty} 2^{m+i} m^k&\| (K_l(\cdot - x) - K_l(\cdot))\chi_{|y|\sim 2^{m+i}} \|_{\A,I_{m+1}}
= 2^{l+i} l^k \left\|\frac1{2^{l+i+1}}\chi_{-F_l\cup F_l}\right\|_{\A,I_{l+1}} \\&\quad+ 2^{l-1+i}(l-1)^k \left\|\frac1{2^{l+i+1}}\chi_{-F_l\cup F_l} + \frac1{2^{l+i}}\chi_{-F_{l-1}\cup F_{l-1}}\right\|_{\A,I_l}\\&\quad+ 2^{l-2+i} (l-2)^k \left\|\frac1{2^l+i}\chi_{-F_{l-1}\cup F_{l-1}}\right\|_{\A,I_{l-1}}
\\&\leq 2^{l+i} l^k \left\|\frac1{2^{l+i+1}}\chi_{-F_l\cup F_l}\right\|_{\A,I_{l+1}} + 2^{l-1+i}(l-1)^k \left\|\frac1{2^{l+i+1}}\chi_{-F_l\cup F_l}\right\|_{\A,I_l}
\\&\quad+ 2^{l-1+i}(l-1)^k\left\|\frac1{2^{l+i}}\chi_{-F_{l-1}\cup F_{l-1}}\right\|_{\A,I_l}
\\&\quad + 2^{l-2+i} (l-2)^k \left\|\frac1{2^{l+i}}\chi_{-F_{l-1}\cup F_{l-1}}\right\|_{\A,I_{l-1}}.
\end{align*}

Using that $\left\|\frac1{2^{l+i+1}}\chi_{-F_l\cup F_l}\right\|_{\A,I_{l}}\leq 2\left\|\frac1{2^{l+i+1}}\chi_{-F_l\cup F_l}\right\|_{\A,I_{l+1}}$ and \\$\left\|\frac1{2^{l+i}}\chi_{-F_{l-1}\cup F_{l-1}}\right\|_{\A,I_{l-1}} \leq 2 \left\|\frac1{2^{l+i}}\chi_{-F_{l-1}\cup F_{l-1}}\right\|_{\A,I_{l}}$ we get,

\begin{align*} \sum_{m=1}^{\infty} &2^{m+i} m^k\| (K_l( \cdot - x) - K_l(\cdot))\chi_{|y|\sim 2^{m+i}} \|_{\A,I_{m+1}}
\\&\leq 2.2^{l+i}l^k\left\|\frac1{2^{l+i+1}}\chi_{-F_l\cup F_l}\right\|_{\A,I_{l+1}} + 2.2^{l-1+i} (l-1)^k  \left\|\frac1{2^{l+i}}\chi_{-F_{l-1}\cup F_{l-1}}\right\|_{\A,I_{l}}
\\&\quad =\frac{l^k}{\A^{-1}\left(\frac{2^{l+i+2}}{2|x|}\right)} + \frac{(l-1)^k}{\A^{-1}\left(\frac{2^{l+i+1}}{2|x|}\right)}
\\&\leq \frac{2l^k}{\A^{-1}\left(\frac{2^{l+i+1}}{|x|}\right)},
\end{align*}
where the last inequality holds due  $\A^{-1}$ is monotone.

Then, for all $ |x|< 2^i$, we obtain,
\begin{align*} \bigg\| \bigg\{ \sum_{m=1}^{\infty} 2^{m+i} m^k\| (K_l(\cdot - x) - K_l(\cdot))\chi_{|y|\sim 2^{m+i}} \|_{\A,t_{m+1}} \bigg\}_{l \in \Z} \bigg\|_{l^2} \leq \Bigg\| \Bigg\{ \frac{2l^k}{\A^{-1}\left(\frac{2^{l+i+1}}{|x|}\right)} \Bigg\}_{l\in\Z}\Bigg\|_{l^2}.
\end{align*}
In particular, the last inequality holds for all $|x|<\frac{2^i}{4}$. As
$ |x|<\frac{2^i}{4} \text{ then } \frac{2^{l+i+1}}{|x|} > 2^l8$,
\begin{align*} \bigg\| \bigg\{ \sum_{m=1}^{\infty} 2^{m+i} &m^k\| (K_l( \cdot - x) - K_l(\cdot))\chi_{|y|\sim 2^{m+i}} \|_{\A,B_{m+1}} \bigg\}_{l \in \Z} \bigg\|_{l^2} \leq \bigg\| \bigg\{ \frac{2l^k}{\A^{-1}\left(\frac{2^{l+i+1}}{|x|}\right)} \bigg\}_{l\in\Z}\bigg\|_{l^2}
\\& \leq \bigg\| \bigg\{ \frac{2l^k}{\A^{-1}\left(2^{l}8\right)} \bigg\}_{l\in\Z}\bigg\|_{l^2} = 2\bigg\| \bigg\{ \frac{l^k}{\A^{-1}\left(2^{l}8\right)} \bigg\}_{l\in\Z}\bigg\|_{l^2}.
\end{align*}
Then, by hypothesis, we obtain $K \in H_{\A,l^2,k}^{\dagger}$.

Now let us prove that  $K \in H_{\A,l^2,k}^{\dagger}$ $\Rightarrow$ $\left\| \left\{ \frac{m^k}{\A^{-1}(2^m8)}\right\}_{m\in\Z}\right\|_{l^2} < \infty.$
By hypothesis, there exist $c_{\A}>1$ and $C_{\A}>0$ such that for all $R \in \mathbb{R}$ and for all $x$, $ |x|c_\A <2^i$, then
\begin{align*} \bigg\| \bigg\{ \sum_{m=1}^{\infty} 2^{m}R \; m^k\| (K_l(\cdot - x) - K_l(\cdot))\chi_{|y|\sim 2^{m}R} \|_{\A,B_{m+1}} \bigg\}_{l \in \Z} \bigg\|_{l^2} \leq C_{\A}.
\end{align*}
Let $i \in \Z$. If $R=2^i$, then $|x|< 2^i$. Thus, using Proposition \ref{defK} we get
\begin{align*} \bigg\| \bigg\{ \sum_{m=1}^{\infty} 2^{m+i} &m^k\| (K_l( \cdot - x) - K_l(\cdot))\chi_{|y|\sim 2^{m+i}} \|_{\A,B_{m+1}} \bigg\}_{l \in \Z} \bigg\|_{l^2} \\& \geq  \biggl\| \biggl\{ 2^{l+i} l^k \left\|\frac1{2^{l+i+1}}\chi_{-F_l\cup F_l}\right\|_{\A,I_{l+1}}  \biggr\}_{l \in \Z} \biggr\|_{l^2}
=  \biggl\| \biggl\{ 2^{l+i} l^k \frac1{2^{l+1+i}\A^{-1}\left(\frac{2^{l+i+2}}{2|x|}\right)} \biggr\}_{l \in \Z} \biggr\|_{l^2}
\\& \quad= \bigg\| \bigg\{ \frac{l^k}{2\A^{-1}\left(\frac{2^{l+i+1}}{|x|}\right)} \bigg\}_{l \in \Z} \bigg\|_{l^2} = \frac1{2}\bigg\| \bigg\{ \frac{l^k}{\A^{-1}\left(\frac{2^{l+i+1}}{|x|}\right)} \bigg\}_{l \in \Z} \bigg\|_{l^2},
\end{align*} this holds for all $|x|<2^i$.
Then, taking supremum, we obtain,
\begin{align*}
C_{\A} &\geq {\underset{2^{i-2}<|x|<2^{i-1}}{\sup}}\bigg\| \bigg\{ \sum_{m=1}^{\infty} 2^{m}R \; m^k\| (K_l( \cdot - x) - K_l(\cdot))\chi_{|y|\sim 2^{m}R} \|_{\A,B_{m+1}} \bigg\}_{l \in \Z} \bigg\|_{l^2}
\\&\geq {\underset{2^{i-2}<|x|<2^{i-1}}{\sup}} \frac1{2}\bigg\| \bigg\{ \frac{l^k}{\A^{-1}\left(\frac{2^{l+i+1}}{|x|}\right)} \bigg\}_{l \in \Z} \bigg\|_{l^2} \geq \frac1{2}\bigg\| \bigg\{ \frac{l^k}{\A^{-1}\left(2^l8\right)} \bigg\}_{l \in \Z} \bigg\|_{l^2}.
\end{align*}
Hence,
\begin{align*} \bigg\| \bigg\{ \frac{l^k}{\A^{-1}\left({2^l8}\right)} \bigg\}_{l \in \Z} \bigg\|_{l^2} \leq 2C_{\A} < \infty.
\end{align*}
\end{proof}

\begin{proof}[Proof of Corollary \ref{prop4}] Let $\A(t)=\exp(t^{1+k})-1$. Using Proposition \ref{prop3}, is enough to prove that for any $k \in \N \cup \{0\}$,
 $$ \bigg\| \bigg\{ \frac{l^k}{\A^{-1}\left({2^l8}\right)} \bigg\}_{l \in \Z} \bigg\|_{l^2} < \infty.$$
As $\A(t)=\exp(t^{1+k})-1$, $\A^{-1}(t)=\log(t+1)^{k+1}$. If $m=0$,
\\ $\A^{-1}(2^m)=\A^{-1}(1)=\log(1+1)^{k+1}=\log(2)^{k+1}\neq 0$, then $\frac{m^k}{\A^{-1}\left({2^m8}\right)}=0$.
Also,
\\ $\A^{-1}(2^m8)=\log(2^m8+1)^{k+1} \geq \log(2^m8)^{k+1} \geq \log(2^m)^{k+1}$.
Then, we get,

\begin{align*} \bigg\| \bigg\{ \frac{m^k}{\A^{-1}\left({2^m8}\right)} \bigg\}_{m \in \Z} \bigg\|_{l^2}^2 &= \sum_{m \in \Z}\bigg(\frac{m^k}{\A^{-1}({2^m8})}\bigg)^2 = \sum_{m \in \Z\setminus \{0\}}\bigg(\frac{m^k}{\A^{-1}({2^m8})}\bigg)^2
\\&\leq \sum_{m \in \Z\setminus \{0\}}\bigg(\frac{m^k}{\log({2^m})^{k+1}}\bigg)^2 = \sum_{m \in \Z\setminus \{0\}}\bigg(\frac1{\log(2)^{k+1}}\frac{m^k}{m^{k+1}}\bigg)^2
\\&\quad =\frac1{\log(2)^{2(k+1)}} \sum_{m \in \Z\setminus \{0\}}\frac1{m^2} < \infty.
\end{align*}

\end{proof}

\section{Proofs of the main results}

 For the proof of the main results we need the following,

\begin{lema}\label{Sharp}
Let $k\in \N \, \cup \,\{0\}$. Let $\A$, $\B$ be Young functions such that $\overline{\A}^{-1}(t)\B^{-1}(t)\overline{\mathcal{C}}_k^{-1}(t) \leq t$, with $\overline{\mathcal{C}}_{k}(t)=\exp(t^{1/k})$ for $t \geq 1$. If $T$ is a  vector-valued singular integral operator with kernel $K$ such that $K \in H_{\B,X}^{\dagger} \cap H_{\A,X,k}^{\dagger}$, then for any $b \in BMO$, $ 0 < \delta < \varepsilon < 1$ we have\\

a) if $k=0$, $\B=\A$, then there exists $C >0$ such that,

$$M_{\delta}^{\sharp} \Vert Tf \Vert_{\mathbf{x}} (x):= \left( M^{\sharp} \Vert Tf \Vert_{\mathbf{x}}^{\delta}  \right)^{\frac{1}{\delta}}(x) \leq  C  \; M_{\overline{A}} f(x),$$
for all $x \in \R$.

b) If $k \in \N$, then there exists $C=C(\delta,\varepsilon) > 0$ such that,
\begin{align*} M_{\delta }^{\#} ( \| T_b^kf\|_X ) (x) \leq C \sum_{j=0}^{k-1} \| b\|_{BMO}^{k-j} M_{\varepsilon}(T_b^jf)(x) + C \| b\|_{BMO}^{k} M_{\overline{\A}}f(x),\end{align*} for all $x \in \R$.

\end{lema}

\begin{proof}

The argument  is similar to the proof of Lemma 5.1 in \cite{LMRT}, we only give the main changes. Let consider the part (b), the part (a) is analog with $k=0$.\\

 Let  $K \in H_{\B,X}^{\dagger} \cap H_{A,X,k}^{\dagger}$ and $k \in \N$.
 Then for any $\lambda \in \mathbb{R}$, we can write
\begin{equation}\label{defT} T_b^kf(x)= T((\lambda - b)^kf)(x) + \sum_{m=0}^{k-1} C_{k,m} (b(x)-\lambda)^{k-m} T_b^mf(x). \end{equation}
Let us fix $x \in \R$ and  $B$ a ball such that $x \in B$, $\tilde{B}:=2B$ and $c_B := \text{center of the ball } B$. Let  $f=f_1 + f_2$, where $f_1:=f \chi_{\tilde{B}}$ and let $a:= \| T(b_{\tilde{B}}- b)^kf_2(c_B) \|_X$.
Using (\ref{defT}) and taking $\lambda = b_{\tilde{B}} = \frac1{|\tilde{B}|} \int_{\tilde{B}} b$,

\begin{align*} \biggl( \frac1{|B|}&\int_B |\|T_b^kf(y)\|_X^{\delta} -|a|^{\delta}| dy \biggr)^{1/\delta} \leq
   \left( \frac1{|B|}\int_B \|T_b^kf(y) - T(b_{\tilde{B}}- b)^kf_2(c_B) \|_X^{\delta} dy \right)^{1/\delta}
\\ & = \Biggl( \frac1{|B|}\int_B \| \sum_{m=0}^{k-1} C_{k,m} (b(y)-b_{\tilde{B}})^{k-m} T_b^mf(y) + T((b_{\tilde{B}} - b)^kf)(y) \\&\qquad\qquad\quad- T((b_{\tilde{B}}- b)^kf_2)(c_B) \|_X^{\delta} dy \Biggr)^{1/\delta}
\\& \leq C \Biggl[\sum_{m=0}^{k-1} C_{k,m}\biggl( \frac1{|B|}\int_B \|(b(y)-b_{\tilde{B}})^{k-m} T_b^mf(y)\|_X^{\delta} dy \biggr)^{1/\delta} \\
& \quad\qquad\qquad+\biggl( \frac1{|B|} \int_{B}\|T((b_{\tilde{B}} - b)^kf_1)(y)\|_X^{\delta} dy \biggr)^{1/\delta} \\
& \quad \qquad\qquad+ \biggl(\frac1{|B|}\int_{B} \| T((b_{\tilde{B}}- b)^kf_2)(y)- T((b_{\tilde{B}}- b)^kf_2)(c_B) \|_X^{\delta} dy \biggl)^{1/\delta} \Biggr]
\\& = C [ I + II + III].
\end{align*}
\\
 The estimates of $I$ and $II$ are analogous to corresponding
in the Lemma 5.1 in \cite{LMRT}. Then
\begin{align*} I  &\leq c \sum_{m=0}^{k-1} C_{k,m} \|b\|_{BMO}^{k-m} M_{\varepsilon}(\|T_b^kf\|_X)(x),\\
 II & \leq C \|b\|_{BMO}^k M_{\overline{\A}}f(x).
\end{align*}

Now $III$. By Jensen's inequality and the property of the norm (\ref{definorma}), we get
\begin{align*} III &\leq \frac1{|B|}\int_{B} \| T((b_{\tilde{B}}- b)^kf_2)(y)- T((b_{\tilde{B}}- b)^kf_2)(c_B) \|_X dy \\&= \frac1{|B|}\int_{B} \biggl\| \biggl\{ \int_{\tilde{B}^c}(K_l(y-z) - K_l(c_B - z))(b_{\tilde{B}}- b(z))^k f(z)dz \biggr\} \biggr\|_X dy
\\& \leq \frac1{|B|}\int_{B} \biggl\| \biggl\{ \biggl|\int_{\tilde{B}^c}(K_l(y-z) - K_l(c_B - z))(b_{\tilde{B}}- b(z))^k f(z)dz \biggr|\biggr\} \biggr\|_X dy
\\& \leq \frac1{|B|}\int_{B} \biggl\| \biggl\{ \int_{\tilde{B}^c}|K_l(y-z) - K_l(c_B - z)||b_{\tilde{B}}- b(z)|^k |f(z)|dz \biggr\} \biggr\|_X dy.
\end{align*}

For each coordinate $l \in \Z$, we proceed as in the proof of Lemma 5.1 in \cite{LMRT}. Let $B_j:=2^{j+1}B$, for $j\geq1$ and we obtain
\begin{align*}\int_{\tilde{B}^c}|K_l(y-z) &- K_l(c_B - z)||b_{\tilde{B}}- b(z)|^k |f(z)|dz\\
&\leq C \|b\|_{BMO}^k M_{\overline{\A}}f(x) \Biggl(\sum_{j=1}^{\infty} (2^jR)^n\|(K_l(y-\cdot) - K_l(c_B - \cdot))\chi_{S_j}\|_{\B,B_j} \\&\qquad \quad + \sum_{j=1}^{\infty} (2^jR)^n j^k \|(K_l(y-\cdot) - K_l(c_B - \cdot))\chi_{S_j}\|_{\A,B_j}\Biggr)
\end{align*}

Hence,
\begin{align*} III &\leq  \frac1{|B|}\int_{B} \biggl\|   \biggl\{C \|b\|_{BMO}^k M_{\overline{\A}}f(x) \Biggl(\sum_{j=1}^{\infty} (2^jR)^n\|(K_l(y-\cdot) - K_l(c_B - \cdot))\chi_{S_j}\|_{\B,B_j} \\&\qquad \quad + \sum_{j=1}^{\infty} (2^jR)^n j^k \|(K_l(y-\cdot) - K_l(c_B - \cdot))\chi_{S_j}\|_{\A,B_j}\Biggr) \biggr\} \biggr\|_X dy
\\ &\leq C \|b\|_{BMO}^k M_{\overline{\A}}f(x) \frac1{|B|}\int_{B} \biggl[ \biggl\| \sum_{j=1}^{\infty} (2^jR)^n\|(K_l(y-\cdot) - K_l(c_B - \cdot))\chi_{S_j}\|_{\B,B_j} \biggr\|_{X} \\&\quad+\biggl\|\sum_{j=1}^{\infty} (2^jR)^n j^k \|(K_l(y-\cdot) - K_l(c_B - \cdot))\chi_{S_j}\|_{\A,B_j}\biggl\|_{X} \biggl] dy
\\ & \leq C \|b\|_{BMO}^k M_{\overline{\A}}f(x) \frac1{|B|}\int_{B} dy =  C \|b\|_{BMO}^k M_{\overline{\A}}f(x),
\end{align*}
where the last inequality holds since $K \in H_{\B,X}^{\dagger} \cap H_{\A,X,k}^{\dagger}$ and we have used that $x \in B \subset B_j$ and that $|x_B- y|< R$ since $y \in B$.
\\

Thus,
\begin{align*} \biggl( \frac1{|B|}\int_B |\|T_b^k(y)\|_X^{\delta} -|a|^{\delta}| dy \biggr)^{1/\delta} \leq C \sum_{m=0}^{k-1} &C_{k,m} \|b\|_{BMO}^{k-m} M_{\varepsilon}(\|T_b^kf\|_X)(x) \\&\quad \qquad + C \|b\|_{BMO}^k M_{\overline{\A}}f(x).
\end{align*}
\end{proof}

Now we proceed to prove the main  theorems.
\begin{proof}[Proof of Theorem \ref{TeoImp}]
Let $w \in A_{\infty},$ and suppose the kernel $\ K \in H_{\A,X}^{\dag}$, where $\A$ is a Young function and $f \in C_c^{\infty}(\R)$.
Let $p>0$, we take $\varepsilon$ \ such that \ $0 < \delta=p \varepsilon < 1.$ Then using the part (a) of Lemma \ref{Sharp} we obtain
$$\begin{aligned}
  \int_{\R} \Vert Tf \Vert_{\mathbf{x}}^{p} w 
&\leq \int_{\R} M_{\varepsilon} \left( \Vert Tf \Vert_{\mathbf{x}}^{p} \right) w
=\int_{\R} \left( M ( \Vert Tf \Vert_{\mathbf{x}}^{p \, \varepsilon})\right)^{\frac{1}{\varepsilon}} w\nonumber\\
&\leq  c\int_{\R} \left( M^{\sharp} (\Vert \Vert Tf \Vert_{\mathbf{x}}^{\delta}) \right)^{\frac{p}{\delta}} \, w
=c\int_{\R} \left( M^{\sharp}_{\delta} (\Vert Tf \Vert_{\mathbf{x}} \right)^{p}w\\
&\leq  c \int_{\R} \left( M_{\overline{A}} f \right)^{p} w,\nonumber
\end{aligned}
$$
 for the second inequality we need the left hand is finite, to prove this we use the fact that $f \in C_c^{\infty}$ imply $\int_{\R} M_{\varepsilon} \left( \Vert Tf \Vert_{\mathbf{x}}^{p} \right) w < \infty.$ (see \cite{D} and \cite{LMRT}).
\\
Thus,
$$\int_{\R} \Vert Tf \Vert_{\mathbf{x}}^{p} w \leq c \int_{\R}\left( M_{\overline{A}} f \right)^{p} w.$$
Since the space $C_c^{\infty}(\R)$ is dense in $L^p(\R)$ for all $p$, we prove the result.

\end{proof}

\begin{proof}[Proof of Theorem \ref{TeoP}]
 The proof is analogous to the proof of Theorem 3.3, part (a), in
\cite{LMRT}, using in this case Lemma \ref{Sharp}.
\end{proof}

\section{Fractional integrals}

For fractional integral operator there exist
$L^{\A,\alpha}$-H\"ormander conditions defined in \cite{BLR}.
The authors obtain the inequality (\ref{Coifineq}) with
$M_{\overline{\A},\alpha}$, the
fractional maximal operator associated to $\overline{\A}$. In this section, we  present
a  weaker condition for fractional vector-valued kernels
 and obtain similar results and applications.\\


Recall the notation: $|x|\sim s$ means $s < |x|\leq 2s$ and given a Young function $\A$ we write $\|f\|_{ \A,|x|\sim s}=\|f \chi_{|x|\sim s}\|_{\A,B(0,2s)}$. \\

The new condition is the following,

\begin{mydef}\label{dagaalpha} Let $K_{\alpha}=\{ K_{\alpha,l} \}_{l \in \Z}$ be a vector-valued function, $\A$ be a Young function, $0<\alpha < n$ and $k \in \N\cup \{0\}$. The function
 $K_{\alpha}$ satisfies the $L^{\alpha,\A,X,k}_{\dagger}$-H\"ormander condition ($K \in H^{\dagger}_{\alpha,\A,X,k}$), if there exist $c_{\A}>1$ and $C_{\A}>0$ such that for all $x$ and $R>c_{\A}|x|$,
 \begin{align*}\Biggl\| \biggl\{ \sum_{m=1}^{\infty} (2^mR)^{n- \alpha} m^k  \| K_{\alpha,l}(\cdot - x) - K_{\alpha,l}(\cdot)\|_{\A,|y|\sim2^mR}\biggr\}_{l\in \Z} \Biggr\|_{X} \leq C_{\A}.\end{align*}
We say that $K_{\alpha} \in H^{\dagger}_{\alpha,\infty,k}$ if $K_{\alpha}$ satisfies the previous condition with $\|\cdot\|_{L^{\infty},|x|\sim 2^mR}$ in place of $\|\cdot\|_{\A,|x|\sim 2^mR}$.

If $k=0$, we denote $H^{\dagger}_{\alpha,\A,X}=H^{\dagger}_{\alpha,\A,X,0}$ and $H^{\dagger}_{\alpha,\infty,X}=H^{\dagger}_{\alpha,\infty,X,0}$.\\
\end{mydef}

Also we need an extra condition that ensure certain control of the size, in this case is,
\begin{mydef} Let $\A$ be a Young function and let $0<\alpha < n$. The function $K_{\alpha} = \{ K_{\alpha,l} \}_{l \in \Z}$ is said to satisfy the ${\mathscr{S}}_{\alpha,\A, X}^{\dagger}$ condition, denote it by $K_{\alpha} \in {\mathscr{S}}_{\alpha,\A, X}^{\dagger}$, if there exists a constant $C>0$ such that
$$\Biggr\|\left\{\|K_{\alpha,l}\|_{\A,|x| \sim s}\right\}_{l \in \Z}\Biggl\|_{X} \leq C s^{\alpha - n}.$$
\end{mydef}

\begin{obs}
If $\A(t)\leq c \B(t)$ for $t>t_0$, some $t_0>0$, then
$$H^{\dagger}_{\alpha,\B,X,k} \subset H^{\dagger}_{\alpha,\A,X,k} \qquad \text{and } \qquad
{\mathscr{S}}_{\alpha,\B, X}^{\dagger}\subset {\mathscr{S}}_{\alpha,\A, X}^{\dagger}. $$

\end{obs}

\begin{obs}
Observe that the $M_{\alpha,\overline{\A}}$ is the fractional maximal operator associated to the Young function $\overline{\A}$, that is
$$ M_{\alpha,\overline{\A}}f(x):= {\underset{B \ni x}{\sup}} |B|^{\alpha/n}\|f\|_{\overline{\A},B}.$$
\end{obs}

 The results in this case are

\begin{teo}\label{Int}Let $\A$ be a Young function and $0< \alpha < n$. Let $T_{\alpha}f=\{ K_{\alpha,l} \ast f\}_{l \in \Z}$ with kernel $K_{\alpha} = \{ K_{\alpha,l} \}_{l \in \Z}  \in {\mathscr{S}}_{\alpha,\A, X}^{\dagger} \cap \, H_{\alpha,\A,X}^{\dagger}$. Let $0 < p < \infty$ and $w \in A_{\infty}$, then there exist $c>0$ such that
\begin{align*} \int_{\R} \| T_{\alpha}f\|^p_X w \leq C \int_{\R} (M_{\alpha,\overline{\A}}f)^p w, \quad f \in L_c^{\infty}(\R),\end{align*}
whenever the left-hand side is finite.
\end{teo}

\begin{teo}\label{teoPalpha}Let $0< \alpha < n$, $b \in BMO$ and $k \in \N$. Let $T_{\alpha}$ convolution operator with kernel $K_{\alpha} = \{ K_{\alpha,l} \}_{l \in \Z}$  such that $T_{\alpha}$ is bounded from $L^{p_0}_{X}(dx)$ to $L^{q_0}_{X}(dx)$, for some $1<p_0,q_0<\infty$. Let $\A$, $\B$ Young function such that  $\overline{\A}^{-1}(t) \B^{-1}(t) \overline{\mathcal{C}}_k^{-1}(t) \leq t$, with $\overline{\mathcal{C}}_{k}(t)=\exp(t^{1/k})$ for $t \geq 1$.
If $K_{\alpha} \in {\mathscr{S}}_{\alpha,\A, X}^{\dagger}\cap H^{\dagger}_{\alpha,\A,X} \cap H^{\dagger}_{\alpha,\B,X,k}$, then for any $0 < p < \infty$ and any $w \in A_{\infty}$, there exist $c>0$ such that
\begin{align*} \int_{\R} \| T_{\alpha,b}^kf\|^p_X w \leq C \|b\|^{pk}_{BMO}\int_{\R} (M_{\alpha,\overline{\A}}f)^p w, \quad f \in L_c^{\infty}(\R),\end{align*}
whenever the left-hand side is finite.
\end{teo}

\begin{obs}
The proof of this results are analogous to the ones in \cite{BLR} with the same changes of the results for the vector-valued singular integral operators above. Also for the proof of this results we need the following lemma and the proof is analogous to the Lemma \ref{Sharp} and the one Theorem 3.6 in \cite{BLR}.
\end{obs}

\begin{lema}
Let $\A$ be a Young function and $0< \alpha < n$. Let $T_{\alpha}f=K_{\alpha} \ast f$ with kernel $K_{\alpha} \in {\mathscr{S}}_{\alpha,\A, X}^{\dagger} \cap \, H_{\alpha,\A,X}^{\dagger}$, then for all $0 < \delta < \varepsilon < 1$ there exists $c>0$ such that
$$M_{\delta}^{\sharp} \| T_{\alpha}f \|_{\mathbf{x}} (x)= \left( M^{\sharp} \| T_{\alpha}f \|_{X}^{\delta} \right)^{\frac1{\delta}}(x) \leq c \; M_{\alpha, \overline{\A}} f(x),$$
for all $x \in \R$ and $f \in L_c^{\infty}$.
\end{lema}

There exist  relations between the kernels which satisfies the fractional conditions ${\mathscr{S}}_{\alpha,\A, X}^{\dagger}$ and $ H^{\dagger}_{\alpha,\A,X,k}$ and the kernels which satisfies the conditions ${\mathscr{S}}_{\A, X}^{\dagger}$ and $ H^{\dagger}_{\A,X,k}$. The next proposition show this relation and also a form to define kernels such that satisfies the fractional condition. The proof is analogous to Proposition 4.1 in \cite{BLR}.

\begin{prop}\label{propSalpha}
Let $K = \{ K_l \}_{l \in \Z}$ and $K_{\alpha} = \{ K_{\alpha,l} \}_{l \in \Z}$ defined by $K_{\alpha}(x)=|x|^{\alpha}K(x)$.
If $K \in {\mathscr{S}}_{\A, X}^{\dagger} \cap H^{\dagger}_{\A,X,k}$ then $K_{\alpha} \in  {\mathscr{S}}_{\alpha,\A, X}^{\dagger} \cap H^{\dagger}_{\alpha,\A,X,k}$.
\end{prop}

We know that, for certain $X$ Banach space, the kernel of the square operator satisfies the conditions ${\mathscr{S}}_{\A,X}^{\dagger}$ and $ H^{\dagger}_{\A,X,k}$, for example $X=l^p$ and $\A(t)=\exp^{\frac1{1+k}}-1$, for more examples see Section 4. Now we can define the fractional square operator,

$$S_{\alpha,X}f(x):=\|\tilde{T}_{\alpha}f(x)\|_{X}=\left\|\left\{\int_{\mathbb{R}} |x-y|^{\alpha}K_l(x-y)f(y)dy\right\}\right\|_{X},$$
where $K=\{K_l\}_{l\in \Z}$ is the kernel defined in the Section 4. Let $b \in BMO$ and  $k \in \N $, the commutator is defined by
$$S_{\alpha,X, b}^kf(x):=\|\tilde{T}_{\alpha,b}^kf(x)\|_{X}=\left\|\left\{ \int_{R}(b(x) - b(y))^k |x-y|^{\alpha} K_l(x-y)f(y) dy \right\}_{l \in \Z}\right\|_{X}.$$

By Proposition \ref{propSalpha}, we have that $S_{\alpha,X}f(x)$ satisfies the hypothesis of Theorem \ref{teoPalpha}. Then,  Theorem \ref{Teo37} for the fractional square operator is

\begin{teo}\label{Sgralalpha} Let $b \in BMO$,  $k \in \N \cup \{0\}$ and $0<\alpha<n$. Let $\A(t)=\exp(t^{\frac1{1+k}})-1$. If
$K \in {\mathscr{S}}_{\A, X}^{\dagger} \cap H_{\A,X,k}^{\dagger}$ i.e. \begin{equation*}\label{condAE3}
  \left\Vert \left\lbrace  \frac1{m} \right\rbrace_{m \in (\Z-\{0\})} \right\Vert_{X}=C_{\A,X}<\infty,
\end{equation*}
then, for any $0<p<\infty$ and $w \in A_{\infty}$, there exists $C$ such that
$$
\begin{aligned}
\int_{\R} \vert S_{\alpha,X, b}^k f(x) \vert^{p} w(x) dx
 \leqslant C \int_{\R} \left(M_{\alpha, L \log L^{k+1}}f (x) \right)^{p} w(x)dx.
\end{aligned}
$$
\end{teo}

In \cite{BDP}, the authors  study the weights for fractional maximal operator related to Young function in the context of variable Lebesgue spaces. They characterized the weights for the boundedness of $M_{\alpha,\A}$ with $\A(t)=t^r(1+\log(t))^{\beta}$, $r\geq 1$ and $\beta \geq 0$.

For any $1\leq p,q <\infty$, we define the $A_{p,q}$ weight class by,
$w\in A_{p,q}$ if and only if $w^q\in A_{1+\frac{q}{p'}}$

The result in the classical Lebesgue spaces, that is the variable Lebesgue spaces with constant exponent, is the following,

\begin{teoA}\cite{BDP}
Let $w$ be a weight, $0<\alpha<n$, $1<p<n/\alpha$, and $1/q=1/p-\alpha/n$. Let $\A(t)=t^r(1+\log(t))^{\beta}$, with $1\leq r < p$ and $\beta\geq 0$. Then $M_{\alpha, \A}$ is bounded from $L^p(w^p)$ into $L^q(w^q)$ if and only if $w^r \in A_{p/r,q/r}$.
\end{teoA}

Applying this result to Theorem \ref{Sgralalpha} we obtain that, if $w \in A_{p,q}$ then for all $1<p<n/\alpha$, and $1/q=1/p-\alpha/n$,
\begin{align*}
\int_{\R} | S_{\alpha,X, b}^k f(x)|^{q} w^q(x) dx & \leq c \int_{\R} \left(M_{\alpha, L \log L^{k+1}}f(x) \right)^{q} w^q(x)dx\\
&\leq c \int_{\R} |f (x)|^{p} w^p(x)dx;
\end{align*}

So we have the following results,

\begin{corol}
Let $0<\alpha<1$, $1<p<1/\alpha$ and $1/q=1/p-\alpha$. If $w\in A_{p,q}$ then $S_{\alpha,X, b}^k$ is bounded from $L^p(w^p)$ into $L^q(w^q)$.
\end{corol}

\end{document}